\documentclass[12pt]{amsproc}

\usepackage{amsmath, amssymb, amsthm, latexsym,nicefrac,setspace}
\usepackage[pagebackref,colorlinks,linkcolor=red,citecolor=blue,urlcolor=blue,hypertexnames=true]{hyperref}
\usepackage{amsrefs}
\usepackage{marginnote}

\usepackage[draft]{fixme}

\theoremstyle{plain}
\newtheorem{theorem}{Theorem}[section]

\newtheorem{lemma}[theorem]{Lemma}

\theoremstyle{remark}

\newtheorem*{theorem*}{Theorem}

\input{xy}
\xyoption{all}

\def\Z{\mathbb Z}

\def\R{\mathbb{R}}

\def\N{\mathbb{N}}

\title{Kazhdan's Property (T) via  semidefinite optimization}

\author{Tim Netzer}
\address{Tim Netzer, Institut f\"ur Geometrie, TU Dresden, 01062 Dresden, Germany}
\email{tim.netzer@tu-dresden.de}

\author{Andreas Thom}
\address{Andreas Thom, Institut f\"ur Geometrie, TU Dresden, 01062 Dresden, Germany}
\email{andreas.thom@tu-dresden.de}

\begin{document}

\onehalfspace

\begin{abstract} Following an idea of Ozawa, we give a new proof of Kazhdan's property (T) for ${\rm SL}(3,\Z)$, by showing that $\Delta^2- \frac{1}{6} \Delta$ is a hermitian sum of squares in the group algebra, where $\Delta$ is the unnormalized Laplace operator with respect to the natural generating set. This corresponds to a spectral gap of $\frac{1}{72}\sim 0.014$ for the associated random walk operator. 

The sum of squares representation was found numerically by a semidefinite programming algorithm, and then turned into an exact symbolic representation, provided in an attached Mathematica file.
\end{abstract}

\maketitle

\section{Introduction}

Eversince David Kazhdan introduced Property (T) in \cite{MR0209390}, it has been a challenge to prove that particular groups have this property or not. Moreover, it has become an intriguing task to provide explicit bounds (in case of particular groups such as ${\rm SL}(n,\mathbb Z)$) for the associated Kazhdan constants or the spectral gap, see \cites{MR1089795 , MR1813225, MR2197816 , MR1779896}. 
See \cite{bekval, MR1813225} for general information on Kazhdan's Property (T).

\vspace{0.2cm}

In \cite{oz}, Narutaka Ozawa has found a new necessary and sufficient criterion of Property (T) for discrete groups, using a sum-of-squares approach inspired by the solution to Hilbert's 17th problem, and a systematic study of its generalizations to a non-commutative context, see for example \cite{schm}. Let $\Gamma$ be a group and $S \subset \Gamma$ be a symmetric finite generating set and define
$$\Delta := |S| - \sum_{s \in S} s.$$ The group $\Gamma$ has Property (T) if and only if $\Delta$ has a spectral gap at zero, i.e., if and only if there exists some $\varepsilon>0$, such that ${\rm sp}(\Delta) \subset \{0\} \cup [\varepsilon, \infty)$ in every unitary representation of $\Gamma$ on a Hilbert space. Ozawa showed in \cite{oz} that this happens if and only if there exists $b_1,\dots,b_n \in \R[\Gamma]$ such that
\begin{equation} \label{eq1}
\Delta^2 - \varepsilon \Delta = \sum_{i=1}^n b_i^* b_i,
\end{equation}
for some $\varepsilon>0$.
Whereas the sufficiency is obvious, the proof of necessity of this criterion involved some observations in the context of so-called non-commutative real algebraic geometry of the group ring associated with the group -- inspired in part on previous work by the authors \cite{netzthom}. Even though sufficiency was easy to establish, Ozawa's work has been an eye-opener and triggered the concrete attempt to find explicit certificates -- that means solutions to Equation \eqref{eq1} -- for groups well-known to have Property (T). In this short paper, we have achieved this goal for the group ${\rm SL}(3,\mathbb Z)$ and its natural generating set consisting of elementary matrices. Using Matlab, Yalmip, SeDuMi and finally Mathematica, we could improve all previous bounds on the spectral gap by a factor of about $2500$. To the best of our knowledge, the best previously known bound for the Kazhdan constant was $1/300$, see \cite[Theorem 12.1.14]{ozawabrown}, and this leads to a spectral gap for unnormalized Laplace operator of at least $1/15000$ --  for example using the estimate from \cite[Proposition 3]{pakzuk}.

In our solution to Equation \eqref{eq1}, we have that $\varepsilon=\frac{1}{6}$, $n \leq 121$, and each $b_i$ a sum of at most $121$ monomials in the real group ring and supported on elements of length $\leq 2$ with respect to the generating set. Even though the length of the elements in the support is rather low and it appears that we have been extremely lucky, it seems unlikely that this sum-of-squares representation (and maybe any other) could have been found without computer assistance.

Our result also shows that the spectral gap is significant in the sense that the resulting bounds on the mixing time of the Product Replacement Algorithm for abelian groups (see \cite{MR1815215} and the references therein) is actually useful for practical purposes.

Any attempt like this for groups like ${\rm Aut}(\mathbb F_4)$ and related groups has failed so far due to lack of computational power  -- and maybe our limited programming skills. Anyhow, the challenge and the desire to understand {\it why} a group has Property (T) and {\it why} the Product Replacement Algorithm for general groups works so well (see \cite{MR1815215}) would remain untouched by any such computational result.

\section{Description of the algorithm}
Let $A$ be a (non-commutative) $*$-algebra over $\R,$ and let $W\subseteq A$ be a finite dimensional subspace. Checking whether some $a\in A$ belongs to $$\Sigma^2 W:=\left\{ \sum_{i=1}^n w_i^*w_i\mid n\in \N, w_i\in W\right\}$$ can be done with the so-called {\it Gram-matrix method}. Fix a basis  $a_1,\ldots,a_m$ of $W$. A  matrix $P\in {\rm M}_m(\R)$ is called a {\it Gram matrix for $a$}, if \begin{equation}\label{gram}a= (a_1^*,\ldots,a_m^*) P  (a_1,\ldots,a_m)^t\end{equation} holds. Note that any $a\in {\rm span}_\R\{ w^*v\mid w,v\in W\}$ admits a Gram matrix, and in general different ones. Indeed the set of all Gram matrices for $a$ is an affine subspace of ${\rm M}_m(\R)$. The element $a$ is {\it hermitian} (i.e. fulfills $a^*=a$) if and only if it admits a symmetric Gram matrix.
The crucial (and straightforward) observation is that $$a\in \Sigma^2W$$if and only if  $a$ admits a {\it positive semidefinite} symmetric Gram matrix. Now this condition can be checked with semidefinite programming. Indeed, equation (\ref{gram}) boils down to  finitely many affine-linear conditions on the entries of $P$, and looking for a positive semidefinite matrix fulfilling these conditions is the feasibility check of a semidefinite programm (see for example \cite{wo}). There exist numerical software for such problems, in our case the Matlab plugin Yalmip \cite{yalmip}, employing the sdp-solver SeDuMi \cite{sedumi}. With these solvers we were able to find a suitable sums of squares representation for our problem.

Let us go into more details. We consider the group $\Gamma={\rm SL}(3,\Z)$ and its real group algebra $A=\R[\Gamma],$ which is a $*$-algebra via $$\left(\sum_g c_g g\right)^*=\sum_g c_g g^{-1}.$$ We use $12$ generators $M_1,\ldots,M_{12}$ for $\Gamma$, obtained from $$M_1=\left(\begin{array}{ccc}1 & 1 & 0 \\0 & 1 & 0 \\0 & 0 & 1\end{array}\right), M_2=\left(\begin{array}{ccc}1 & 0 & 1 \\0 & 1 & 0 \\0 & 0 & 1\end{array}\right),  M_3=\left(\begin{array}{ccc}1 & 0 & 0 \\0 & 1 & 1 \\0 & 0 & 1\end{array}\right)$$
 by closing up under transposing and taking inverses. We then form all products of length $\leq 2$ in these generators. After removing multiples, this results in $m=121$ group elements, including the identity matrix. We let $W$ be the linear span of these group elements in $\R[\Gamma]$ and take $a_1,\ldots,a_m\in W$ to be the natural basis. Now let $$\Delta:= 12 -\sum_{i=1}^{12} M_i=\frac12 \sum_{i=1}^{12}(1-M_i)^*(1-M_i)\in A$$ be the {\it unnormalized} Laplace operator, and set $$a:=\Delta^2- \epsilon \Delta$$ with $\epsilon =0.2805.$   We used the Matlab plugin Yalmip with sdp-solver SeDuMi to find a positive semidefinite matrix $P \in M_m(\R)$ satisfying (\ref{gram}) in the group algebra $\R[\Gamma]$. From the numerical result $P$ we computed a square-root, rounded to rational entries,  and multiplied with a common denominator $10^6$. The resulting integer matrix is stored in the attached file "RootOfP.txt". 
  
 Now certifying that the result is correct is done as follows (and can be reproduced with the attached Mathematica \cite{wolf} file "Sl3ZComment.nb"). We take the matrix $Q$ stored in "RootOfP.txt" and change it slightly to make all rows sum to zero.  We then compute $P:=\frac{1}{10^{12}}Q^tQ$. Thus $P$ is a positive semidefinite matrix with rational entries, and the total sum over all entries equals zero. With $a_1,\ldots,a_m$ as above we compute the sum of squares $$b:=(a_1^*,\ldots,a_m^*) P (a_1,\ldots,a_m)^t$$ in the group algebra. Since the sum over all entries  of $P$ is zero, $b$ will lie in the augmentation ideal $$\omega[\Gamma]:=\left\{ \sum_g c_g g\mid \sum_g c_g=0\right\}\subseteq \R[\Gamma].$$ We next compare $b$ to $a=\Delta^2- \epsilon \Delta.$ The difference $c=a-b$ is hermitian and  belongs again to the augmentation ideal $\omega[\Gamma]$. Furthermore, each group element in the support of $c$ is a product of at most four of our generators $M_i$, since each $a_i$ is a product of at most two of them. We compute $\Vert c\Vert_1\leq 0.0225$. From Lemma \ref{estimate}  below applied when $d=2$ we know that $c+ 0.09 \cdot \Delta$ is a sum of squares, and so is 
 
 $$a+0.09\Delta=b + c + 0.09 \Delta=\Delta^2 - (\epsilon-0.09)\Delta.$$ 
 
 Since $\epsilon-0.09\geq 0.19 \geq \frac16$, this finishes the proof. Note that all computations in this final step are done with Mathematica and involve rational numbers only -- and are therefore exact.
The missing lemma is a quantitative version  of Lemma 2 in \cite{oz}:
\begin{lemma}\label{estimate}Let $\Gamma$ be a group with finite generating set $S=S^{-1}$, and let $$\Delta= \vert S\vert -\sum_{s\in S} s\in \R[\Gamma]$$ be the unnormalized Laplace operator. Let $c=\sum_g c_g g\in \omega[\Gamma]^h$ be such that whenever $c_g\neq 0$, then $g$ is a product of at most $2^d$ elements from $S$. Then $$c+ 2^{2d-1}\Vert c\Vert_1\cdot  \Delta\in \Sigma^2 \R[\Gamma],$$    where $\Vert c\Vert_1=\sum_g \vert c_g\vert.$ If $S$ does not contain self-inverse elements, then even $$c+ 2^{2d-2}\Vert c\Vert_1\cdot  \Delta\in \Sigma^2 \R[\Gamma].$$
\end{lemma}
\begin{proof} From the proof of Lemma 2 in \cite{oz} we use the equation $$(1-gh)^*(1-gh) \leq 2(1-g)^*(1-g) + 2(1-h)^*(1-h)$$ for group elements $g,h$, where $\xi \leq\eta$ means that $\eta-\xi$ is a sum of squares. By iteration we immediately obtain \begin{equation}\label{iter}(1-g_1\cdots g_{2^d})^*(1-g_1\cdots g_{2^d} )\leq 2^d \sum_{i=1}^{2^d} (1-g_i)^*(1-g_i)\end{equation} for group elements $g_i$.
For $c=\sum_g c_g g\in \omega[\Gamma]^h$ we have $\sum_g c_g=0$ and $c_g=c_{g^{-1}}$ for all $g$. Thus $$-c=\sum_{g\neq e} \frac{c_g}{2} \left(2-g-g^{-1}\right)=\sum_{g\neq e} \frac{c_g}{2}(1-g)^*(1-g).$$  Since every occuring $g$ is a product of at most $2^d$ elements from $S$, and $$(1-s)^*(1-s)\leq 2\Delta$$ for $s\in S$, we get from (\ref{iter}) $$-c\leq \sum_{g\neq e, c_g> 0} \frac{c_g}{2}(1-g)^*(1-g)\leq  \left(\sum_{g\neq e,c_g> 0} c_g \right)  2^{2d} \Delta.$$  From $\sum_g c_g=0$ we obtain $\sum_{g\neq e, c_g>0} c_g \leq \frac12 \Vert c\Vert_1$, the result. In case that $S$ does not contain self-inverse elements, we even have $(1-s)^*(1-s)=2-s-s^{.1}\leq \Delta.$ This yields the improved statement.
\end{proof}

\section*{Acknowledgments}

This research was supported by ERC-StGr No.\ 277728 "Geometry and Analysis of Group Rings".

\end{document}